\documentclass{amsart}

\usepackage{tikz}
\usepackage{hyperref}
\usetikzlibrary{cd}

\newtheorem{theorem}{Theorem}[section]

\newtheorem{corollary}[theorem]{Corollary}

\theoremstyle{definition}

\newtheorem{definition}[theorem]{Definition}
\newtheorem{example}[theorem]{Example}

\theoremstyle{remark}
\newtheorem{remark}[theorem]{Remark}

\numberwithin{equation}{section}

\newcommand{\N}{\ensuremath{\mathbb{N}}}   
\newcommand{\R}{\ensuremath{\mathbb{R}}}   
\newcommand{\norm}[1]{\left\lVert #1\right\rVert}   
\newcommand{\Hol}{\ensuremath{\mathrm{Hol}}}   
\newcommand{\GL}{\ensuremath{\mathrm{GL}}}   
\newcommand{\SO}{\ensuremath{\mathrm{SO}}}   
\newcommand{\U}{\ensuremath{\mathrm{U}}}   
\newcommand{\SU}{\ensuremath{\mathrm{SU}}}   
\newcommand{\Sp}{\ensuremath{\mathrm{Sp}}}   
\newcommand{\End}{\ensuremath{\mathrm{End}}}   
\newcommand{\Nabla}[1][]{\ensuremath{\nabla^{#1}}}   
\newcommand{\sTens}{\ensuremath{\mathrm{S}}}   
\newcommand{\Loop}{\ensuremath{\Omega_{\textnormal{p-}C^1}}}   
\newcommand{\CSC}{\ensuremath{\mathrm{CSC}}}   
\newcommand{\Met}{\ensuremath{\mathrm{Met}}}   

\begin{document}

\title{Monotonicity of Holonomy Groups}


\author[L. G{\"o}tzfried]{Linus G{\"o}tzfried}
\address{L. G{\"o}tzfried, Universit{\"a}t Regensburg, 93053 Regensburg, Germany}
\email{linus.goetzfried@gmx.de}
\thanks{The author would like to thank Bernd Ammann for valuable supervision and comments.}

\subjclass[2020]{Primary 53C29; Secondary 53C05;}

\date{\today}


\begin{abstract}
We prove the following monotonicity result for the holonomy group: Given a sequence of metric connections converging in $C^0$ such that all its members have holonomy contained in a closed group $H$, also their limit connection needs to have holonomy contained in $H$. As a corollary, for a sequence of Riemannian metrics converging in $C^1$ and having special restricted holonomy, their limit metric must also have special restricted holonomy. In particular, this implies that the map assigning to Riemannian metrics on a manifold the conjugacy classes of their restricted holonomy groups is lower semicontinuous with respect to the order relation given by inclusion of representatives.
\end{abstract}

\maketitle

\section{Introduction}

An important notion in differential geometry is the one of \emph{holonomy}. Given a connection on a principal or vector bundle, one can define the holonomy group of the connection as the group of endomorphisms of the fibres given by parallel translation along piecewise-differentiable based loops. Of particular interest are the holonomy groups of the Levi-Civita connection on a Riemannian manifold, which are subgroups of the special orthogonal group. In many cases, a ``special'' holonomy of a Riemannian manifold implies interesting consequences for the resulting metric. For example, a smooth $4m$-dimensional Riemannian manifold with holonomy contained in the group $\Sp(m)\Sp(1)$ is automatically an Einstein manifold (see e.\thinspace g.\ \cite{besse_einstein_2008}, Theorem~14.39).

In this work, we study the dependence of the holonomy group on the connection. We focus on the case of metric connections on a vector bundle and prove the following monotonicity result:
\begin{theorem}
\label{thm:HolonomyLimit}
Let $M$ be a connected manifold of class $C^2$, let $x\in M$ and $E\to M$ be an orientable real $C^1$-vector bundle with standard fibre $\R^l$. Let $g\in C^1(M,\sTens^2E^*)$ be a bundle metric on $E$ and $\Nabla[g]$ be a $C^0$-connection on $E$ that is compatible with $g$. Moreover, for $k\in\N$, let $g_k\in C^1(M,\sTens^2 E^*)$ be a bundle metric on $E$ and $\Nabla[k]$ be a $C^0$-connection compatible with $g_k$, such that $g_k|_x\to g|_x$ in $\sTens^2 E^*_x$ and $\Nabla[k]\to\Nabla[g]$ in $C^0$.

Let $H\subseteq\SO(l)$ be closed. Assume that for all $k\in\N$, there exists a positively oriented $g_k$-orthonormal basis of $E_x$ such that with respect to this basis, we have $\Hol_x(\Nabla[k])\subseteq H$. Then, there exists a positively oriented $g$-orthonormal basis of $E_x$ such that with respect to this basis, we have $\Hol_x(\Nabla[g])\subseteq H$.
\end{theorem}

Here, $\sTens^2 E^*$ is the symmetric tensor product of $E^*$ with itself and $C^0$-convergence of connections is defined via the compact-open topology on $C^0(T^* M\otimes E\otimes E^*)$, see Section \ref{sec:ProofTheorem1.1}. Thus, it is shown that in the case of $C^0$-convergence of metric connections, the holonomy can only become more special and not more generic.

In particular, this finds an application for Riemannian holonomy groups. Given a manifold $M$, denote by $\Met(M)$ be the set of all $C^1$-Riemannian metrics on $M$, equipped with the compact-open $C^1$-topology. We obtain:
\begin{corollary}
\label{cor:HolonomyLimit}
Let $(M,g)$ be an oriented $n$-dimensional Riemannian manifold of class $C^2$, let $g\in\Met(M)$ and for $k\in\N$, let $g_k\in\Met(M)$. Let $H\subseteq\SO(n)$ be a closed subgroup and let $x\in M$. Assume that $g_k\to g$ in $\Met(M)$ and for all $k\in\N$, we have $\Hol_x(g_k)\subseteq H$ up to conjugation in $\SO(n)$. Then, $\Hol_x(g)\subseteq H$ up to conjugation in $\SO(n)$.
\end{corollary}

These statements hold true for the full holonomy group and a fortiori also for the restricted one. However, for the latter, we can show more. Let $\CSC(\SO(n))$ be the set of closed subgroups of $\SO(n)$ modulo the conjugation equivalence relation; it will be shown in Section \ref{sec:Application} that this is a partially ordered set with respect to inclusion of representatives. Then, with respect to this order relation, the above monotonicity result implies the following:
\begin{theorem}
\label{thm:HolonomyLowerSemicont}
Let $M$ be an oriented $n$-dimensional manifold of class $C^2$ and let $x\in M$. The map
\[\Met(M)\to \CSC(\SO(n)),\qquad g\mapsto[\Hol_x^0(g)],\]
is lower semicontinuous.
\end{theorem}
An appropriate definition of lower semicontinuity of a map from a topological space into a partially ordered set is given in Section \ref{sec:Application}.

The motivation and some of the ideas of this paper are taken from the recent preprint \cite{muller_restricted_2024} by M{\"u}ller. However, the proof of the main theorem is different from the cited preprint, and the resulting theorem is stronger (note that \cite{muller_restricted_2024} requires $C^2$-regularity of the involved metrics).

\section{Situation and proof of Theorem~\ref{thm:HolonomyLimit}}
\label{sec:ProofTheorem1.1}

We recall the definition of the holonomy group.
\begin{definition}
Let $M$ be a connected manifold of class $C^2$, let $x\in M$ and let $E\to M$ be an orientable real $C^1$-vector bundle over $M$ with standard fibre $\R^l$, $l\in\N$.
\begin{enumerate}
\item We denote by $\Loop(x)$ the set of all piecewise-differentiable loops\\ $c:[0,1]\to M$ based at $x$.
\item Let $\Nabla$ be a connection on $E$ and $c\in\Loop(x)$. We define the \emph{parallel transport} $P_c^{\Nabla}$ along $c$ (with respect to $\Nabla$) as follows: Given any $e\in E_x$, there exists a unique section $s$ of $E|_{c([0,1])}$ with $\Nabla_{\dot c(t)}s(t)=0$ for all $t\in[0,1]$, and we let $P_c^{\Nabla}(e):=s(1)$. Then

\begin{minipage}{0.9\textwidth}
\[\Hol_x(\Nabla):=\left\{P_c^{\Nabla}\,\middle|\,c\in\Loop(x)\right\}\]
\end{minipage}
forms a subgroup of $\GL(E_x)$, the \emph{holonomy group} of $\Nabla$ (at $x$).
\item Given a connection $\Nabla$ on $E$, we define its \emph{restricted holonomy group} at $x$,

\begin{minipage}{0.9\textwidth}
\begin{align*}
\Hol_x^0(\Nabla):=&\,\left\{P_c^{\Nabla}\,\middle|\,c\in\Loop(x)\textnormal{ such that $c$ is contractible}\right\}\\\subseteq&\,\Hol_x(\Nabla).
\end{align*}
\end{minipage}
Both $\Hol_x(\Nabla)$ and $\Hol_x^0(\Nabla)$ are immersed Lie subgroups of $\GL(E_x)$. It can be shown (see e.\thinspace g.\ \cite{ambrose_theorem_1953}, Section~2, Lemma~4) that $\Hol_x^0(\Nabla)$ is the connected component of the identity in $\Hol_x(\Nabla)$.
\item Given two connections $\Nabla,\tilde\Nabla$ on $E$ with continuous coefficients, their difference $\Nabla-\tilde\Nabla$ is a continuous section of $T^*M\otimes E\otimes E^*$. We equip the space $C^0(M,T^*M\otimes E\otimes E^*)$ of such sections with the compact-open $C^0$-topology (see Appendix~\ref{app:CompactOpen}) and say that a sequence of connections $\Nabla[k]$ on $E$ converges in $C^0$ to a connection $\Nabla$ if $\Nabla[k]-\Nabla\to 0\in C^0(M,T^*M\otimes E^*\otimes E)$.
\item Let $h$ be a metric on $E$ and $\Nabla[h]$ be a $C^0$-connection on $E$ that is compatible with $h$, i.\thinspace e.\ $X( h(v,w))=h(\Nabla[h]_X v,w)+h(v,\Nabla[h]_X w)$ for all $X\in TM$, $v,w\in C^1(M,E)$. We denote by

\begin{minipage}{0.9\textwidth}
\begin{align*}
\SO_{l}(E_x)_h:=\{&A\in\End(E_x)\,|\,\det(A)=1,\\
&\forall v,w\in E_x: h|_x(Av,Aw)=h|_x(v,w)\}
\end{align*}
\end{minipage}
the special orthogonal group of $E_x$ with respect to the metric $h|_x$. Then, since $\Nabla[h]$ is metric, we have $P_c^{\Nabla[h]}\in\SO_{l}(E_x)_h$ for all $c\in\Loop(x)$.
\item Choosing any basis of $E_x$, we may identify $E_x$ with $\R^l$ and consequently view $\Hol_x(\Nabla)$ and $\Hol_x^0(\Nabla)$ for any $\Nabla$ as subgroups of $\GL(l)$. However, this is only up to conjugation; the choice of different bases yields subgroups conjugate in $\GL(l)$.
\item Let $g\in C^1(M,\sTens^2E^*)$ be the bundle metric on $E$ from the statement of Theorem~\ref{thm:HolonomyLimit}. We choose a $g$-orthonormal basis $(b_1,\dots,b_l)$ of $E_x$ and use it to identify $E_x$ with $\R^l$, $\SO_{l}(E_x)_g$ with $\SO(l)$ and consequently $\Hol_x(\Nabla[g])$ with a subgroup of $\SO(l)$.
\end{enumerate}
Now, given a (possibly different) metric $h\in C^1(M,\sTens^2 E^*)$, we may still view $\SO_{l}(E_x)_h$ as a distinguished subgroup of $\GL(l)$, as follows:
\begin{enumerate}
\label{def:NonstandardEmbedding}
\item[(8)] The basis $(b_1,\dots,b_l)$ allows us to associate to any bilinear form $Q$ on $E_x$ its representing matrix $M_{Q}\in\R^{l\times l}$; if $Q$ is non-degenerate, $M_Q$ is invertible. Given $h\in C^1(M,\sTens^2 E^*)$ with $M_{h|_x}\in B_1(M_{g|_x})$, we define

\begin{minipage}{0.9\textwidth}
\[\SO(l)_h:=\{A\in\GL(l)\,|\,A^TM_{h|_x}A=M_{h|_x}\}\]
\end{minipage}
(the $x$-dependency is understood). Let $W_h$ be a symmetric positive definite square root of $M_{h|_x}$ as defined by the power series of $\sqrt{1+c}$ in $B_1(0)$; then the map $\SO(l)\to\SO(l)_h,A\mapsto W_h^{-1}AW_h$ is an isomorphism. We call it the \emph{nonstandard embedding} $\phi_h:\SO(l)\to\SO(l)_h\subset\GL(l)$. For a sequence $(h_k)_{k\in\N}\subset C^1(M,\sTens^2E^*)$ with $M_{h_k|_x}\in B_1(M_{g|_x})$ for all $k$ and $M_{h_k|_x}\to M_{g|_x}$ in $\R^{l\times l}$, we have $W_{h_k}\to I_l$ (the $l\times l$-identity matrix).
\end{enumerate}
\end{definition}

Now, we can prove the first announced result, Theorem~\ref{thm:HolonomyLimit}.
\begin{proof}[Proof of Theorem~\ref{thm:HolonomyLimit}]
Let $(b_1,\dots,b_l)$ be the positively oriented $g$-orthonormal basis of $E_x$; we use it to identify $E_x$ with $\R^l$ and $H$ with a subgroup of $\SO(l)$ as before. Then, the conclusion of the theorem may be reformulated as follows: There exists $U\in\SO(l)$ such that $U\Hol_x(\Nabla[g])U^{-1}\subseteq H$.

We may assume that the representing matrices of $g_k|_x$ and $g|_x$ satisfy $M_{g_k|_x}\in B_1(M_{g|_x})$ for all $k\in\N$ by dropping a finite number of elements from the sequence, hence we may define the matrices $W_{g_k}$, the subgroups $\SO(l)_{g_k}$ and the nonstandard embedding $\phi_{g_k}$ for $k\in\N$ as in Definition \ref{def:NonstandardEmbedding}. Let us denote by $H_{g_k}$ the image of the composition of the inclusion $H\hookrightarrow\SO(l)$ and the nonstandard embedding $\SO(l)\to\SO(l)_{g_k}$.

The identification $E_x=\R^l$ translates the assumptions of the theorem into the assertion that for all $k\in\N$ there exists $U_k\in\SO(l)_{g_k}$ with $U_k\Hol_x(\Nabla[k])U_k^{-1}\subseteq H_{g_k}$. Now, the sequence $\phi_{g_k}^{-1}(U_k)$ is contained in the compact group $\SO(l)$ and hence possesses a convergent subsequence. Moreover, using the explicit definition of $\phi_{g_k}$, we see that if $(\phi_{g_{k_m}}^{-1}(U_{k_m}))_{m\in\N}$ converges to some $U\in\SO(l)$, then also $(U_{k_m})_{m\in\N}$ converges to $U$.

Therefore, passing to the subsequence without changing the notation for brevity, we may assume that $(U_k)_{k\in\N}$ converges to some $U\in\SO(l)$. It now only remains to prove the following: \newline

\emph{Claim.} We have $U\Hol_x(\Nabla[g])U^{-1}\subseteq H$.

\emph{Proof of the claim.} It is sufficient to show that for any $c\in\Loop(x)$, we have $P_c^{\Nabla[g]}\in H$. By definition of the compact-open topology, the restrictions of the connections to the (compact) image of $c$ converge in $C^0$. Therefore, the coefficients of the parallel transport equation defining $P_c^{\Nabla[k]}$ converge uniformly to the coefficients of the parallel transport equation defining $P_c^{\Nabla[g]}$. From well-known facts about ordinary differential equations it follows that after possibly again passing to a subsequence, we may assume that $\lim\limits_{k\to\infty}P_c^{\Nabla[k]}=P_c^{\Nabla[g]}$. (See e.\thinspace g.\ \cite{hartman_ordinary_2002}, Chapter~I, Theorem~2.4.) Thus, also
\[UP_c^{\Nabla[g]}U^{-1}=\lim\limits_{k\to\infty}U_kP_c^{\Nabla[k]}U_k^{-1}=\lim\limits_{k\to\infty}W_{g_k}U_kP_c^{\Nabla[k]}U_k^{-1}W_{g_k}^{-1}.\]
Since $U_kP_c^{\Nabla[k]}U_k^{-1}\in H_{g_k}$ respectively $W_{g_k}U_kP_c^{\Nabla[k]}U_k^{-1}W_{g_k}^{-1}\in H$ for all $k\in\N$, we obtain $UP_c^{\Nabla[g]} U^{-1}\in H$ by continuity of matrix multiplication. This proves the claim and hence the theorem.
\end{proof}

\begin{remark}
It shall be stressed that the proof of the claim above makes no assertion about the uniformity of the convergence $P_c^{\Nabla[k]}\to P_c^{\Nabla[g]}$ with respect to the loop $c$, but such one is also not required.
\end{remark}

\begin{remark}
The theorem also holds with the holonomy groups replaced by the restricted holonomy groups, with the same proof.
\end{remark}

\section{Application to Riemannian holonomy groups}
\label{sec:Application}
Theorem~\ref{thm:HolonomyLimit} may be applied to and refined in the case of oriented Riemannian manifolds and their Levi-Civita connections. 
\begin{definition}
Let $M$ be a manifold of class $C^2$. Let $\Met(M)$ be the set of all $C^1$-Riemannian metrics on $M$, equipped with the compact-open $C^1$-topology. Given $h\in\Met(M)$ with Levi-Civita connection $\Nabla[h]$ and $x\in M$, we denote $\Hol_x(h):=\Hol_x(\Nabla[h])$ and $\Hol_x^0(h):=\Hol_x^0(\Nabla[h])$.
\end{definition}

The proof of Corollary~\ref{cor:HolonomyLimit} is straightforward.
\begin{proof}[Proof of Corollary~\ref{cor:HolonomyLimit}]
By the Koszul formula, $C^1$-convergence $g_k\to g$ implies $C^0$-convergence $\Nabla[g_k]\to\Nabla[g]$. The corollary follows from Theorem \ref{thm:HolonomyLimit}.
\end{proof}

As a further corollary, we get:
\begin{corollary}
Let $M$ be an oriented $n$-dimensional Riemannian manifold of class $C^2$, let $H\subseteq\SO(n)$ be a closed subgroup and let $x\in M$. Then, the set
\begin{align*}
\{h\in \mathrm{Met}(M)\,|\,&\Hol_x(h)\subseteq H\textnormal{ up to conjugation}\}
\end{align*}
is closed in $\mathrm{Met}(M)$.
\end{corollary}

For the restricted holonomy group, the results can also be phrased in terms of lower semicontinuity of the map $\Met(M)\to\CSC(\SO(n))$, $g\mapsto[\Hol_x^0(g)]$ as stated before, see Theorem~\ref{thm:HolonomyLowerSemicont}. The proof will follow again as an application of Theorem~\ref{thm:HolonomyLimit}, once it has been shown that the statement is actually well-defined.

\begin{definition} 
\label{def:semicont}
$\ $
\begin{enumerate}
\item Let $X$ be a topological space and let $(Y,\leq)$ be a partially ordered set. We say that a function $f:X\to Y$ is \emph{lower semicontinuous at $x\in X$} if there exists a neighbourhood $U$ of $x$ such that $f(x)\leq f(x')$ for all $x'\in U$. We say that $f:X\to Y$ is \emph{lower semicontinuous}, if it is lower semicontinuous at all $x\in X$.
\item Let $G$ be a compact Lie group and let $\CSC(G)$ be the set of closed subgroups of $G$ modulo the conjugation equivalence relation

\begin{minipage}{0.9\textwidth}
\[A\sim B:\Leftrightarrow \exists_{g\in G}: gAg^{-1}=B.\]
\end{minipage}
Let ``$\leq$'' be the relation on $\CSC(G)$ defined by inclusion of representatives, that is

\begin{minipage}{0.9\textwidth}
\[[A]\leq[B]:\Leftrightarrow \exists_{A\in [A], B\in [B], g\in G}: gAg^{-1}\subseteq B.\]
\end{minipage}
\end{enumerate}
\end{definition}

The relation ``$\leq$'' on the set $\CSC(G)$ defined as above (for $G$ any compact Lie group) will be shown to define a partial order on $\CSC(G)$; hence we obtain a notion of lower semicontinuous functions to $\CSC(G)$. To prove antisymmetry, the following Cantor-Bernstein-like theorem is needed, which we learnt from \cite{muller_restricted_2024}.

\begin{theorem}
\label{thm:CantorBernsteinLieGroupConjugation}
$\ $
\begin{enumerate}
\item Let $K,H$ be Lie groups with finitely many connected components. Let $f:K\to H$ and $g:H\to K$ be injective Lie group homomorphisms. Then $f$ and $g$ are Lie group isomorphisms.
\item Let $K,H$ be closed Lie subgroups of a compact Lie group $G$. If there exist $g,g'\in G$ such that $gKg^{-1}\subseteq H$ and $g'H{g'}^{-1}\subseteq K$, then in fact $gKg^{-1}=H$ and $g'H{g'}^{-1}=K$.
\end{enumerate}
\end{theorem}
We include a proof of the theorem in order to make the article self-contained, moreover we give a slightly different argument for the case of multiple connected components compared to \cite{muller_restricted_2024}.
\begin{proof}
\emph{Ad 1.} Let $K,H$ be Lie groups with finitely many connected components. Let $f:K\to H$ and $g:H\to K$ be injective Lie group homomorphisms. We prove that $f$ is a Lie group isomorphism, the proof for $g$ is analogous.

We start out by proving: \newline

\emph{Claim.} The linear map $d_1 f:T_1 K\to T_1 H$ is an isomorphism.

\emph{Proof of the claim.} First, we show injectivity. It is well known that for the exponential maps $\exp_K$ and $\exp_H$ of $K$ resp. $H$ we have $\exp_H\circ d_1 f=f\circ\exp_K$. Furthermore, both exponential maps are local diffeomorphisms, in particular injective if restricted to a suitable neighbourhood of zero in the respective tangent spaces.

Now let $v\in T_1 K$ with $d_1 f(v)=0$. Let $U\subseteq T_1 K$ be a neighbourhood of zero such that $\exp_K|_U$ is a diffeomorphism onto its image. In particular, $(f\circ\exp_K)|_U$ is injective by assumption on $f$. Let furthermore $\lambda\in(0,\infty)$ be so small that $w:=\lambda v\in U$. By linearity we have $d_1 f(w)=0$. Thus
\[f(\exp_K(w))=\exp_H(d_1 f(w))=\exp_H(0)=1=f(\exp_K(0))\]
and therefore $w=0$ by injectivity of $(f\circ\exp_K)|_U$. Hence, also $v=\lambda^{-1}w=0$ and we conclude that $d_1 f$ is injective.

Now, analogously, one shows that $d_1 g:T_1 H\to T_1 K$ is injective. By dimension counting it follows that both maps must in fact be isomorphisms, proving the claim.\newline

Let $H_0,K_0$ be the connected components of $1\in H_0$ and $1\in K_0$, respectively. It is well known that $\exp_H$ surjects onto $H_0$; moreover the claim directly states that $d_1 f$ is surjective. Thus $f\circ\exp_K=\exp_H\circ d_1 f$ surjects onto $H_0$ and since $\mathrm{im}(\exp_K)=K_0$, one concludes that $H_0\subseteq f(K_0)$. Hence $f|_{K_0}:K_0\to H_0$ is an isomorphism. The map $f$ induces a map $\overline{f}:K/K_0\to H/H_0$ which fits into the following commutative diagram with exact rows:
\[\begin{tikzcd}
K_0 \arrow[r] \arrow[d, "f|_{K_0}"] & K \arrow[r] \arrow[d,"f"] & K/K_0 \arrow[r] \arrow[d,"\overline{f}"] & 1 \arrow[d] \\
H_0 \arrow[r] & H \arrow[r] & H/H_0 \arrow[r] & 1
\end{tikzcd}\]
(the horizontal maps are the obvious ones). In this diagram, the first vertical map is therefore surjective and the second and fourth vertical maps are injective. By one of the four-lemmas, the third vertical map is hence injective. Analogously, one proves that the map $\overline{g}:H/H_0\to K/K_0$ induced by $g$ is injective. Now $H/H_0$ and $K/K_0$ are in bijection with the sets of connected components of $H$ and $K$, thus they are finite. Therefore, $\overline{f}$ and $\overline{g}$ must be isomorphisms.

Thus, for every $h\in H$, there exist $k_1\in K$ and $h_0\in H_0$ such that $h=f(k_1)h_0$. On the other hand, due to surjectivity of $f|_{K_0}:K_0\to H_0$, there exists $k_0\in K_0$ with $h_0=f(k_0)$, and then $h=f(k_1)f(k_0)=f(k_1 k_0)$ lies in the image of $f$. One concludes that $f$ is surjective, hence bijective.

Now, to finish the proof, it suffices to recall the well-known statement that a bijective Lie group homomorphism is already a Lie group isomorphism. \newline

\emph{Ad 2.} Let $K,H$ be closed Lie subgroups of a compact Lie group $G$ and $g,g'\in G$ with $gKg^{-1}\subseteq H$ and ${g'}H{g'}^{-1}\subseteq K$.We can view $K,H$ as Lie groups for their own right; since they are closed in the compact set $G$, they are compact as well. In particular, they have only finitely many connected components. Now part 2 follows from part 1, applied to the injective Lie group homomorphisms $K\to H,k\mapsto gkg^{-1}$ and $H\to K,h\mapsto g'h{g'}^{-1}$.
\end{proof}

\begin{corollary}
Let $G$ be a compact Lie group. The relation ``$\leq$'' on $\CSC(G)$ from Definition~\ref{def:semicont} is a partial order on $\CSC(G)$.
\begin{proof}
Reflexivity of the relation ``$\leq$'' is clear. Transitivity follows from the fact that if there exist closed subgroups $A,B,B',C$ of $G$ and $g_1,g_2,g\in G$ with $g_1 A g_1^{-1}\subseteq B$, $gBg^{-1}=B'$ and $g_2 B' g_2^{-1}\subseteq C$, then $(g_2 g g_1) A (g_2 g g_1)^{-1}\subseteq C$. Finally, antisymmetry of the relation ``$\leq$'' follows directly from Theorem \ref{thm:CantorBernsteinLieGroupConjugation}, part~2.
\end{proof}
\end{corollary}

Now, we can prove Theorem~\ref{thm:HolonomyLowerSemicont}.
\begin{proof}[Proof of Theorem~\ref{thm:HolonomyLowerSemicont}]
It is well known (see e.\thinspace g.\ \cite{joyce_compact_2000}, Theorem~3.2.8) that for any $h\in\Met(M)$, the group $\Hol_x^0(h)$ is closed in $\SO(n)$. Hence the map from the statement of the theorem is actually well-defined. Now the statement follows from Theorem~\ref{thm:HolonomyLimit}.
\end{proof}

Theorem~\ref{thm:HolonomyLimit} and Theorem~\ref{thm:HolonomyLowerSemicont} are in a certain sense rigidity results: As already stated in the introduction, the holonomy can only become more special, not more generic. However, full continuity of the map from Theorem~\ref{thm:HolonomyLowerSemicont} is in general not true. For example, for Ricci-flat deformations of a Ricci-flat metric the conjugacy class of the holonomy group is known to be constant if the universal covering of $M$ admits a parallel spinor (see \cite{ammann_holonomy_2019}). In more general cases, one can easily construct explicit counterexamples.

\begin{example}
Consider subsets of scaled Poincaré disk models, $M:=B_{1/2}(0)\subset\R^2$ equipped with the family of metrics $g_k(x^1,x^2):=\frac{4}{(1-((x^1)^2+(x^2)^2)/k^2)^2}\,dx^1\otimes dx^2$ in cartesian coordinates. These converge in $\Met(M)$ to the limit metric $g(x^1,x^2):=4\,dx^1\otimes dx^2$.

For any $k\in\N$, $(M,g_k)$ is up to scaling describable as a subset of the hyperbolic space of dimension $2$. Hence, as is well known, $\Hol_x(g_k)=\Hol_x^0(g_k)\cong\SO(2)$ for all $x\in M$ (see e.\thinspace g.\ \cite{besse_einstein_2008}, Proposition~10.79 and Chapter~10.K, Table~3). However $\Hol_x^0(g)=\{I_2\}\subset\SO(2)$ because $g$ is flat.
\end{example}

One may also perturb a Calabi-Yau metric without leaving the realm of K{\"a}hler metrics, that is, to obtain a sequence of metrics with holonomy $\U(n)$ converging to a metric with holonomy $\SU(n)$:
\begin{example}
Let $M$ be any $\mathrm{K}3$ surface and $\omega$ be the K{\"a}hler form of a K{\"a}hler metric on $M$.  Let $\rho\in C^\infty(M,\Lambda^{1,1}T^* M)$ with $[\rho]=0$ in cohomology (the existence of such a $\rho$ follows directly from the fact that $C^\infty(M,\Lambda^{1,1}T^* M)$ is infinite-dimensional while $H^2(M)$ is not).

By Yau's solution to the Calabi conjecture in \cite{yau_ricci_1978}, there exists a unique K{\"a}hler metric $g_1$ on $M$ with Ricci form $\rho$ and such that its K{\"a}hler form is contained in the cohomology class of $\omega$. Moreover, there exists a unique Ricci-flat K{\"a}hler metric $g_\infty$ on $M$ with K{\"a}hler form contained in the cohomology class of $\omega$.

Let $g_k:=\frac{1}{k}g_1+\frac{k-1}{k}g_\infty$ for $k\in\N_{\geq 2}$. Then, the metrics $g_k$ are K{\"a}hler since convex combinations of K{\"a}hler metrics are again K{\"a}hler, and the sequence $(g_k)_{k\in\N}$ converges in $\Met(M)$ to $g_\infty$. Moreover, for all $k\in\N$ the metric $g_k$ is not Ricci-flat (if it were, it would contradict the uniqueness statement from the Calabi conjecture).

Thus, the metrics $g_k$ all have holonomy group contained in $\U(2)$, but not in $\SU(2)$ (see e.\thinspace g.\ \cite{joyce_compact_2000}, Proposition~4.4.2 and Proposition~6.1.1). However, $g_\infty$ is Ricci-flat and hence has its holonomy group contained in $\SU(2)\subset\U(2)$.
\end{example}

As a final example, it shall be noted that at least on incomplete manifolds, one can also obtain quaternion-K{\"a}hler-metrics (with holonomy in $\Sp(m)\Sp(1)$) converging to a hyperk{\"a}hler metric (with holonomy in $\Sp(m)$). This is due to Swann (\cite{swann_hyperkahler_1990}).
\begin{example}
\label{ex:SwannBundle}
Let $N$ be a quaternion-K{\"a}hler manifold of dimension $4m$ with positive scalar curvature. Let $M$ be the Swann bundle over $N$, which is an incomplete manifold fibring over $N$ with fibre $(\mathbb{H}\backslash\{0\})/\mathbb{Z}_2$ that has been constructed in \cite{swann_hyperkahler_1990}, Section~2.1. In Theorem~2.1.7 of \cite{swann_hyperkahler_1990}, a family of Riemannian quaternion-K{\"a}hler metrics $g_k$ with nonzero scalar curvature on $M$ is constructed, as well as a hyperk{\"a}hler metric $g_\infty$, such that $g_k\to g_\infty$ in $\Met(M)$ (the latter statement follows from the explicit descriptions given there). By well-known facts about the scalar curvatures of quaternion-K{\"a}hler and hyperk{\"a}hler metrics (see e.\thinspace g.\ \cite{besse_einstein_2008}, Theorem~14.45) it follows that $[\Hol_x^0(g_\infty)]\leq[\Sp(m+1)]$ is strictly less than all $[\Hol_x^0(g_k)]$ with respect to the order relation on $\CSC(\SO(4m+4))$ from Definition \ref{def:semicont}. Consequently, the map from Theorem \ref{thm:HolonomyLowerSemicont} is not upper semicontinuous at $g_\infty$.
\end{example}
On the other hand, it would be impossible to construct an example as the previous one on a closed manifold: On a compact manifold $M$ of dimension $n$, all Einstein metrics are stationary points of the Einstein-Hilbert functional $S:\Met(M)\to\R$, $g\mapsto\frac{\int_M \mathrm{scal}^g\,\mathrm{dvol}^g}{\left(\int_M\mathrm{dvol}^g\right)^{(n-2)/n}}$. In particular, since $S$ needs to remain constant along any deformation within the space of Einstein metrics, it is impossible to deform Ricci-flat Einstein metrics (such as hyperk{\"a}hler ones) on a compact manifold to non-Ricci-flat Einstein metrics (such as quaternion-K{\"a}hler ones). Contrarily, Theorem~\ref{thm:HolonomyLowerSemicont} holds without a compactness assumption on the manifold.

To conclude this paper, it shall be remarked that one would assume that the conditions on Theorem~\ref{thm:HolonomyLimit} can actually be weakened. The most apparent restriction is the one to connections compatible with positive definite metrics, which makes the theorem, for example, inapplicable in the Lorentzian setting. Moreover, it would be interesting to see whether Theorem~\ref{thm:HolonomyLowerSemicont} also holds for the full holonomy group (which is, by \cite{wilking_compact_1999}, in general not closed in the special orthogonal group).

However, the relaxation of these conditions would require different proofs. It would be interesting to study such generalizations.

\appendix
\section{The compact-open $C^r$-topology}
\label{app:CompactOpen}
In order for this article to be self-contained, we recall here the definition of the compact-open $C^r$-topology (see e.\thinspace g.\ \cite{hirsch_differential_1976}, Chapter~2.1).
\begin{definition}
Let $r\geq 0$, $M$ be a $C^{r+1}$-manifold and $F\to M$ be a vector bundle of class $C^r$. Given any two charts $\mathbf{x}:U\to V,\mathbf{v}:U'\to V'$ of $M$ respectively $F$, a compact $K\subseteq U$, $\epsilon>0$ and $f\in C^r(M,F)$, let

\begin{align*}B(f,\mathbf{x},\mathbf{v},K,\epsilon):=\Big\{&g\in C^r(M,F)\,|\,g(K)\subseteq V,\forall 0\leq k\leq r:\\
&\norm{D^k(\mathbf{v}\circ f\circ\mathbf{x}^{-1})-D^k(\mathbf{v}\circ g\circ\mathbf{x}^{-1})}_{C^0}<\epsilon\Big\}.
\end{align*}
The set of all such $B(f,\mathbf{x},\mathbf{v},K,\epsilon)$ has the basis property and thus defines a topology on $C^r(M,F)$, the \emph{compact-open $C^r$-topology}.

If a sequence of tensors $(T_k)_{k\in\N}\subset C^0(M,F)$ converges in the compact-open $C^0$-topology to a tensor $T\in C^0(M,F)$, then for any compact $K\subseteq M$, the restrictions $T_k|_K$ converge to $T|_K$ uniformly on $K$ (with respect to any set of charts covering $K$).
\end{definition}

\bibliographystyle{amsplain}
\bibliography{bibliography.bib}

\end{document}